\documentclass[10pt,a4paper]{article}
\usepackage[utf8]{inputenc}
\usepackage[T1]{fontenc}

\usepackage{amsmath,amssymb,amsthm}
\usepackage{diagbox}
\usepackage{tikz}
\usetikzlibrary{arrows,positioning} 
\usepackage{longtable}
\usepackage{multirow}
\usepackage{array}
\usepackage{colonequals}
\usepackage{graphicx}
\usepackage{geometry}
\usepackage{setspace}
\usepackage{enumerate}
\usepackage[colorlinks,
  linkcolor=blue,
  citecolor=blue,
  filecolor=blue,
  menucolor=blue,
  urlcolor=blue,
  pdfnewwindow=true,
  pdftitle ={Realizability\ and\ inscribability\ for\ simplicial\ polytopes\ via\ nonlinear\ optimization},
  pdfauthor= {Moritz\ Firsching}, pdfsubject={Realizability\ and\ inscribability\ for\ simplicial\ polytopes\ via\ nonlinear\ optimization}]{hyperref}
\DeclareMathOperator{\sign}{sign}

\newcommand{\overbar}[1]{\mkern 1.5mu\overline{\mkern-0.5mu#1\mkern-0.5mu}\mkern 1.5mu}

\theoremstyle{definition}
\newtheorem{definition}{Definition}
\theoremstyle{plain}
\newtheorem{theorem}[definition]{Theorem}
\newtheorem{conjecture}[definition]{Conjecture}
\newtheorem{proposition}[definition]{Proposition}

\theoremstyle{remark}
\newtheorem{question}[definition]{Question}

\newtheorem{construction}[definition]{Construction}

\newcommand\numberthis{\addtocounter{equation}{1}\tag{\theequation}}

\tikzset{
    >=stealth',
    punkt/.style={
           rectangle,
           rounded corners,
           draw=black, very thick,
           text width=6.5em,
           minimum height=2em,
           text centered},
    pil/.style={
           ->,
           thick,
           shorten <=2pt,
           shorten >=2pt,}
}
\begin{document}
 \title{Realizability and inscribability for simplicial polytopes via nonlinear optimization}
 \author{Moritz Firsching\thanks{Supported by the DFG within SFB/TRR 109 ``Discretization in Geometry and Dynamics''}\\
\small Institut für Mathematik \\[-0.8ex]
\small Freie Universität Berlin\\[-0.8ex]
\small Arnimallee 2\\[-0.8ex]
\small 14195 Berlin\\ [-0.8ex]
\small Germany\\[-0.3ex]
\small \href{mailto:firsching@math.fu-berlin.de}{firsching@math.fu-berlin.de}
}

\date{\today}

\maketitle
\begin{abstract}
 We show that nonlinear optimization techniques can successfully be applied to realize and to inscribe matroid polytopes and simplicial spheres.
 Thus we obtain a complete classification of neighborly polytopes of dimension $4$, $6$ and $7$ with $11$ vertices, of neighborly $5$-polytopes with $10$ vertices, 
 as well as a complete classification of simplicial $3$-spheres with $10$ vertices into polytopal and non-polytopal spheres. 
 Surprisingly many of the realizable polytopes are also inscribable. 
\end{abstract}
\tableofcontents
\section{Introduction}
\subsection{Previous results}
The classification of polytopes has been a major goal in discrete geometry. 
Euclid's elements culminates in the last proposition of its last book in which Euclid remarks that there are precisely five regular polyhedra \cite[Liber XIII, Propositio 18]{E82}.
The study of polytopes in higher dimensions was started in the middle of the 19th century, see for example the work by Schläfli \cite{S01}, Wiener \cite{W64} and Schlegel \cite{S86}. 
The enumeration of polytopes with a fixed number of faces emerged as a question and Eberhardt \cite{E91} gave some answers, see Brückner \cite{B00}.

We are interested in the classification of polytopes up to combinatorial type.
Two polytopes are combinatorially equivalent if they have isomorphic face lattices. Given an arbitrary lattice $L$, we can ask: is $L$ polytopal, i.e.~the face lattice of a polytope?
Other than classifying \emph{all} polytopes of dimension $d$ with $n$ vertices we will focus on the subfamilies of \emph{simplicial} and \emph{simplicial neighborly} polytopes. 
These families are of particular interest:
a polytope is simplicial if all of its facets are simplices. If all the vertices of a polytope are in general position, then it is simplicial.
It is simplicial neighborly if the number of its $i$-dimensional faces is maximized for all $i$ among all polytopes with a fixed number of vertices by McMullen's Upper Bound Theorem \cite{M70}.
The first known neighborly polytopes were the cyclic polytopes;
Motzkin conjectured that the cyclic polytopes are the only combinatorial types of neighborly polytopes, which turned out to be false, see \cite{M57} and \cite[p. 225 and \S2]{G63}.
In even dimensions all neighborly polytopes are simplicial, while in odd dimensions there are neighborly but non-simplicial polytopes. 
(For example, every $3$-polytope is neighborly.)

The classification of combinatorial types of $d$-polytopes on $n$ vertices for $n\leq d+3$ was achieved by using Gale diagrams, see \cite[Sect. 6.1-3]{G67} and \cite[Sect. 6.5]{Z07}. 
There are formulae for the number of combinatorial types:
\begin{itemize}
 \item For $n=d+2$ vertices there are $\lfloor d^2/4\rfloor$ combinatorial types of polytopes,  
$\lfloor d/2\rfloor$ combinatorial types of simplicial polytopes and there is only one neighborly polytope: the cyclic polytope, see \cite[Sect. 6.1]{G67}. 
 \item  For $n=d+3$ vertices an erroneous formula for the number of $d$-polytopes with $n$ vertices was given by Lloyd \cite{L70}; it has been corrected by Fusy \cite[Th. 1]{F06}, see \href{https://oeis.org/A114289}{A114289}.
For simplicial $d$-polytopes with $n$ vertices there is a formula by Perles \cite[Sect. 6.2, Th. 6.3.2, p. 113 and p. 424]{G67}, see Bagchi and Datta, \cite[Rem. 6 (C)]{BD98}, \href{https://oeis.org/A000943}{A000943}.
There is a formula for the number of neighborly and simplicial neighborly $d$-polytopes with $d+3$ vertices, see the work by McMullen and Altshuler \cite{Mc74}, \cite{AM73} and \href{http://oeis.org/A007147}{A007147}.
The number of simplicial neighborly $(2n-3)$-polytopes with $2n$ vertices coincides with the number of self-dual $2$-colored necklaces with $2n$ beads and it is possible to find a simple bijection between these two combinatorial objects.
Similarly, there is also a relation between the self-dual $2$-colored necklaces on $2n$ beads and simplicial $(n-3)$-polytopes with $n$ vertices. This is provided by Montellano-Ballesteros and Strausz \cite{MS04}.
For the corresponding classification of $(d-1)$-dimensional polytopal spheres with $d+3$ vertices see Mani \cite{M72} for the simplicial case and Kleinschmidt \cite{K76} for the general case.
\end{itemize}

\begin{table}[h!]
\begin{center}
\begin{tabular}{r|r|r|r|r|r|r|r|r|r|r|}
\backslashbox{$d$}{$n$}&$4$&$5$&$6$&$7$&$8$&$9$&$10$&$11$&$12$\\\hline
$d=3$ all  &$1$ &$2$&$7$&$34$&$257$&$2\,606$&$32\,300$&$440\,564$&$6\,384\,634$\\
 simplicial&$1$&$1$&$2$&$5$&$14$&$50$&$233$&$1\,249$&$7\,595$\\
+neighborly&$1$&$1$&$2$&$5$&$14$&$50$&$233$&$1\,249$&$7\,595$\\\hline
$d=4$ all  && $1$ &$4$  &$31$  &\begin{tabular}{@{}r@{}}$1\,294$\\\cite{AS85}\end{tabular}  &? &? &?&?\\
 simplicial&& $1$ &$2$  &$5$  &\begin{tabular}{@{}r@{}}$37$\\\cite{GS67}\end{tabular} &\begin{tabular}{@{}r@{}}$1\,142$\\\cite{ABS80}\end{tabular} &$\mathbf{162\,004}$ &?&?\\
 neighborly&& $1$ &$1$  &$1$  &\begin{tabular}{@{}r@{}}$3$\\  \cite{GS67}\end{tabular} &\begin{tabular}{@{}r@{}}$23$\\\cite{AS73}\end{tabular} &\begin{tabular}{@{}r@{}}$431$\\\cite{A77}\end{tabular}&$\mathbf{13\,935}$&\begin{tabular}{@{}c@{}}$\mathbf{\geq556\,061}$\\$\mathbf{\leq556\,062}$\end{tabular}\\\hline
$d=5$ all  &&&$1$  &$6$  &$116$      &\begin{tabular}{@{}r@{}}$47\,923$\\\cite{FMM13}\end{tabular} &? &?&?\\
 simplicial&&&$1$  &$2$  &$8$      &\begin{tabular}{@{}r@{}}$322$\\\cite{FMM13}\end{tabular}&? &?&?\\
+neighborly&&&$1$  &$1$  &$2$ &\begin{tabular}{@{}r@{}}$126$\\\cite{Fin14}\end{tabular}&$\mathbf{159\,375}$ &?&?\\\hline
$d=6$ all  &&&     &$1$  &$9$      &$379$     &?&?&?\\
 simplicial&&&     &$1$  &$3$      &$18$     &?&? &?\\
 neighborly&&&     &$1$  &$1$      &$1$ &\begin{tabular}{@{}r@{}}$37$\\\cite{BS87}\end{tabular}&$\mathbf{42\,099}$&? \\\hline
$d=7$  all &&&     &     &$1$      &$12$     &$1\,133$             &? &?     \\
 simplicial&&&     &     &$1$      &$3$     &$29$             &?  &? \\
+neighborly&&&     &     &$1$      &$1$     &$4$   &$\mathbf{35\,993} $  &?   \\\hline
$d=8$ all  &&&     &     &         &$1$       &$16$		&$3\,210$				&?\\
 simplicial&&&     &     &         &$1$       &$4$		&$57$				&?\\
 neighborly&&&     &     &         &$1$       &$1$		&$1$				&\begin{tabular}{@{}r@{}}$2\,586$\\\cite{MP15}\end{tabular}\\\hline
$d=9$ all  &&&     &     &         &       &$1$		&$20$				&$8\,803$\\
 simplicial&&&     &     &         &        &$1$		&$4$				&$96$\\
+neighborly&&&     &     &         &       &$1$		&$1$				&$5$\\\hline
$d=10$ all &&&     &     &         &       &		&$1$				&$25$\\
 simplicial&&&     &     &         &       &		&$1$				&$5$\\
 neighborly&&&     &     &         &       &		&$1$				&$1$\\\hline\hline
 any $d$, all&$2$&$4$&$13$&$73$&$1\,677$&?&?&?&?\\
simplicial  &$2$&$3$&$6$&$14$&$64$&$1\,537$&?&?&?\\
+neighborly &$2$&$3$&$5$&$9$&$22$&$203$&$\mathbf{160\,083}$&?&?\\\hline
\end{tabular}
\caption{Families of $d$-polytopes with $n$ vertices that have been enumerated. \textbf{Boldface} results are new. 
For $n\leq d+3$  results are due to \cite{G67} \cite{AM73} \cite{Mc74} \cite{F06}. For the row $d=3$, see \cite{B88} \cite{T80} \cite{RW82}.\label{tab1}.} 
\end{center}
\end{table}

The most important result in the classification of polytopes of dimension $3$ is Steinitz's theorem: 
the face lattices of (simplicial) $3$-polytopes are in bijection with the $3$-connected (cubic) planar graphs with at least $4$ vertices, see \cite[Satz 43, p. 77]{S22}.
The asymptotic behavior of the number of combinatorial types (simplicial) $3$-polytopes with $n$ vertices is known precisely, see \cite{B88} \cite{T80} \cite{RW82} and \href{https://oeis.org/A000944}{A000944}, \href{https://oeis.org/A000109}{A000109}.

For higher dimensional polytopes much less is known. In fact there is no hope to find a criterion in terms of local conditions as in the $3$-dimensional case, see \cite{S87} and \cite{K88}. 
Already in dimension $4$, deciding whether a given face lattice is the face lattice of a polytope is is in fact complete for the  ``existential theory of the reals'', see Richter-Gebert and Ziegler \cite{RZ95}. 
This problem is known to be NP-hard and also the problem of determining whether an orientable matroid is realizable is known to be NP-hard, see the results by Mn\"ev and Shor \cite{M88,S91} 
even if restricted to neighborly polytopes, by a result of Adiprasito, Padrol and Theran \cite{APT15}.
For an overview of known realization algorithms see \cite[A.5, p. 486]{BVSWZ99}.

However, complete enumerations/classifications have been achieved for some pairs $(d,n)$ with $n\geq d+4$ and $d\geq 4$. The first attempt was an enumeration of simplicial $4$-polytopes with $8$-vertices by Brückner \cite{B09}. 
A mistake in his enumeration was fixed by the complete classification of this family by Grünbaum and Sreedharan \cite{GS67}, thereby also classifying neighborly $4$-polytopes with $8$-vertices. 
It also provided the first examples of non-cyclic neighborly polytopes. 
This was followed by a few results by Altshuler, Bokowski and Steinberg \cite{AS73} \cite{A77} \cite{ABS80} \cite{AS85}, all for $4$-dimensional polytopes, until the classification of neighborly 
$6$-polytopes with $10$ vertices by Bokowski and Shemer \cite{BS87}. 
After that no significant progress in the classification of these families of polytopes has been made for a long time. 
Closely related to these classifications are the classifications of oriented matroids. A good summary of results in that area is \cite{F01}. 
Very recently Fukuda, Miyata and Moriyama \cite{FMM13} classified various families of oriented matroids and obtained classification of $5$-polytopes with $9$ vertices.
Miyata and Padrol \cite{MP15} classified neighborly $8$-polytopes with $12$ vertices.

Table \ref{tab1} summarizes known and new enumeration results of families of $d$-polytopes on $n$ vertices.

\subsection{Our contributions}
We propose a new algorithmic approach in order to give complete enumeration results for simplicial $3$-spheres with $10$ vertices and for various families of neighborly polytopes. 
We not only provide a complete description in rational coordinates, but give realizations with all vertices on the unit sphere if possible, thereby proving inscribability for many polytopes.
We also provide two further applications: the classification of simplicial $3$-spheres with small valence and a special realization of the Bokowski--Ewald--Kleinschmidt polytope.
We hope that our results might be used as a treasure trove of examples and potential counterexamples for the study of polytopes.  
By polar duality our enumeration results for families of (inscribable) \emph{simplicial} polytopes imply the results on corresponding families of  (circumscribable) \emph{simple} polytopes.

\begin{figure}[h]
\begin{center}
\begin{tikzpicture}
 \node[punkt] (simpolys) {simplicial polytopes};  
 \node[punkt, right =2.5cm of simpolys] (uniformoms) {uniform matroid polytopes}
  edge[<-,thick] node[auto=right] {$M$}   (simpolys)
  edge[->, thick, bend left=60,dashed] node[auto=right] {bfp} (simpolys)
  edge[->, thick, bend left=60,dashed] node[auto=left, xshift=.3cm] {Sect. \ref{oms}} (simpolys)
  edge[->, thick, bend right=60] node[auto=left,yshift=-.1cm] {optimization}(simpolys)
  edge[->, thick, bend right=60] node[auto=right, xshift=.3cm] {Sect. \ref{opti}}(simpolys);
 \node[punkt, right =2.5cm of uniformoms] (simspheres) {simplicial spheres}
   edge[<-,thick]node[auto=right] {$S$}  (uniformoms)
   edge[->, thick, bend right=60] node[yshift=-.4cm,xshift=0cm] {mpc \dots} (uniformoms)
   edge[->, thick, bend right=60] node[auto=right,xshift=-.35cm] {Sect. \ref{sstoom}} (uniformoms)
   edge[->, thick, bend left=60,dashed] node[yshift=.4cm, xshift=-.1cm] {\dots IP, SAT} (uniformoms)
   edge[->, thick, bend left=60,dashed] node[auto=left,xshift=-.4cm] {Sect. \ref{sstoom}} (uniformoms)
  edge[->, thick, bend left=60,dashed] node[yshift=.4cm] {bfp on partial chirotope} (simpolys)
  edge[->, thick, bend left=60,dashed] node[auto=left] {Sect. \ref{partchiro}} (simpolys)
   edge[->, thick, bend right=60]node[auto=left] {optimization} (simpolys)
   edge[->, thick, bend right=60]node[auto=right] {{Sect. \ref{opti}}} (simpolys);
  \node[above  =2.7cm of uniformoms] {\emph{finding preimages}};
  \node[below  =2.7cm of uniformoms] {\emph{proving {non-existence} of preimage}};
\end{tikzpicture}
\end{center}
\caption{simplicial polytopes, uniform matroid polytopes and simplicial spheres.}\label{fig1}
\end{figure}
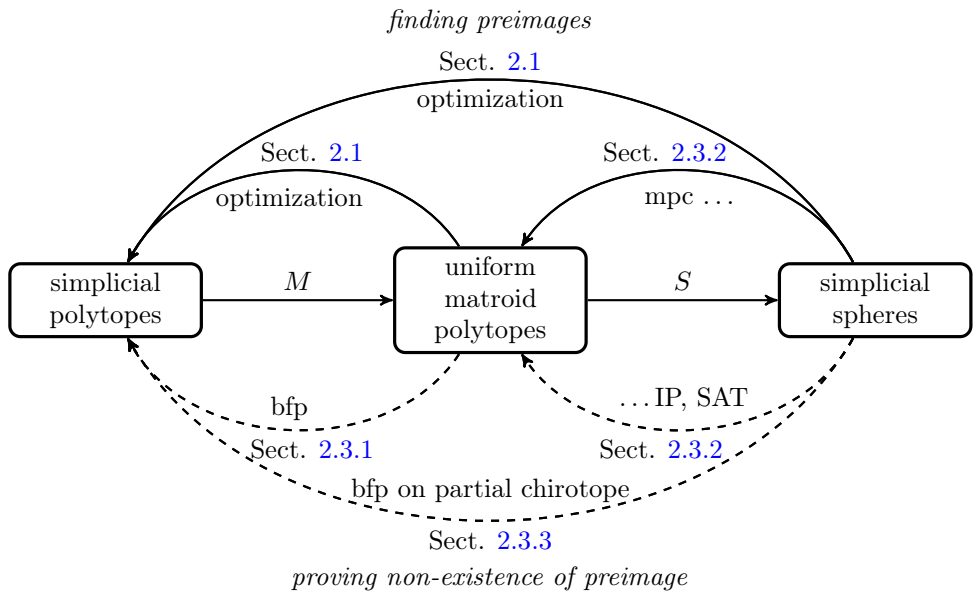A simplicial polytope with vertices in general position yields a uniform matroid polytope by taking the induced oriented matroid. 
(Remember that every combinatorial type of simplicial polytope has a realization with its vertices in general position).
We denote this map from simplicial polytopes to uniform matroid polytopes by $M$.  
In turn a uniform matroid polytope gives rise to a simplicial sphere, by taking the face lattice.
We denote this map from uniform matroid polytopes to simplicial spheres by $S$. 
The composition $S\circ M$ maps a simplicial polytope to the simplicial sphere induced by its boundary.
A uniform matroid polytope is realizable if and only if is has a preimage under the map $M$. 
A simplicial sphere is polytopal if and only if it has a preimage under the map $S\circ M$, i.e.~if it has a preimage under the map $S$ that is a realizable uniform matroid polytope.
For the map $S$ there is a technique that can provide definite results in both directions: to either prove the non-existence of preimages or to find preimages. 
This technique checks the consistency of the chirotope axioms, which are binary constraints, and is discussed in Section \ref{sstoom}. 
For the maps $M$ and $S\circ M$ we use different methods for proving the non-existence of preimages and finding preimages under these maps: 
\begin{description}
 \item[finding preimages:] Here we employ nonlinear optimization software in order to solve systems of nonlinear inequalities. This is explained in Section \ref{opti}. 
The realizations are first obtained numerically and then converted to rational realizations, such that the combinatorial type can be checked using exact arithmetic; see Section \ref{rational}.
 \item[proving the non-existence of preimages:] For proving non-realizability we rely on classical methods of finding final polynomials such as finding biquadratic final polynomials (bfp), see \cite[Sect. 7.3, p. 121]{BS89} and apply this also to partial chirotopes, see Section \ref{nonreal}
\end{description}

We combine these techniques with previous classification results: For the enumeration of some families of neighborly polytopes we build on the enumeration of corresponding families of neighborly uniform oriented matroids given by Miyata and Padrol \cite{MP15} \cite{M}
For the enumeration of simplicial $4$-polytopes with $10$ vertices and $4$-polytopes with small valence we build on the enumeration of corresponding simplicial spheres by Lutz \cite{L08} \cite{FLS15} \cite{L}.

The methods presented here can not only be used to realize a combinatorial type of simplicial polytope, but also the combinatorial type of a uniform matroid polytope or even a non-uniform oriented matroid. 
In the resulting point configurations additional methods would have to be used in order to obtain results in exact arithmetic and not only numerical results, since the methods presented in 
Section \ref{rational} would fail in this case. It is also possible to consider other (simplicial) manifolds than the sphere and obtain realizations. 
The optimization approach also allows for additional requirements on the objects being realized, we exemplify this with inscribability:
In many cases and without much additional difficulty we could find realizations on the sphere, proving inscribability of the polytopes in question. 
For all the families of neighborly polytopes we enumerated, there was not a single non-inscribable case, which leads us to the belief that there might be none, see Conjecture \ref{crazy}.
Since we are free to choose an objective function, we can also find extremal realizations. 
We don't need to focus on combinatorial type, but can consider other equivalence classes of polytopes. 
An example where we optimize over all polytopes similar to a given polytope can be found in \cite{Firsching15a}.
\subsection{Remaining questions and a brave conjecture}
 Already Steiner \cite[Question 77), p. 316]{S32} asked if all polytopes are inscribable. 
 One reason why ``inscribable'' is an interesting property of a polytope is the close relationship with Delaunay triangulations and Voronoi diagrams, as provided by \cite{B79}, see also Section \ref{BEK}.
 For $3$-dimensional polytopes, the situation is well understood. There is a characterization of inscribable $3$-polytopes by Hodgson, Rivin and Smith \cite{HRS92}
 and there are conditions for inscribability on the edge-graph given by Dillencourt and Smith \cite{DS96}, which can be checked algorithmically.
 A criterion for the inscribability of \emph{stacked} polytopes that also works in higher dimensions is given by Gonska and Ziegler \cite[Th. 1]{GZ13}. Having a complete edge graph is far away from being stacked.
The question whether all (even-dimensional) neighborly polytopes are inscribable has been asked by Gonska and Padrol in \cite[p. 2]{GP15}.
 All the neighborly polytopes in the families that we have enumerated are inscribable. Although this is all the evidence we have, we propose a conjecture.
 
\begin{conjecture}\label{crazy}
  All $2$-neighborly simplicial polytopes are inscribable.
\end{conjecture}

Out of the simplicial $4$-polytopes with $10$ vertices, which we have enumerated, we can decide inscribability in all but $13$ cases. We expect the remaining cases to be non-inscribable.
\begin{question}[see Theorem \ref{th23} \ref{zwei})]
  Are the remaining $13$ cases non-inscribable?
\end{question}
\begin{question}
  Is there an efficient method for proving non-inscribability of combinatorial types of non-stacked polytopes of dimension greater than $3$?
\end{question}

For two families we are able to realize all but a single combinatorial type. We expect the answers to the following questions to be negative:
\begin{question}[see Theorem \ref{4vert} \ref{556061})]
 Is there a $4$-dimensional polytope on $12$ vertices with facet list $F374225$? Is it inscribable?
\end{question}
\begin{question}[see Theorem \ref{thsmvalence}]
 Is there a $4$-dimensional polytope on $14$ vertices with facet list $T2775$? Is it inscribable?
\end{question}

\section{Methods for finding realizations and proving non-realizability}\label{polys}
    \setcounter{subsection}{-1}
 \subsection{Definitions and notations}
  We quickly revise the basics of polytopes, oriented matroids and chirotopes. Our notation mostly coincides with those outlined in \cite{BVSWZ99}, \cite[Sect. 6]{RGZ04}, \cite{BS89} and the introduction of \cite{MP15}.
  
    \begin{definition}[face lattice, combinatorial equivalence, neighborly]
    The set of faces of a $d$-polytope partially ordered by inclusion is the \emph{face lattice}. 
    The number $m$-dimensional faces of a $d$-dimensional polytope $P$ is denoted by $f_m(P)$ or $f_m$ and the vector $(f_0,f_1,\dots,f_d)$ is the \emph{$f$-vector} of $P$.
    A $d$-polytope is \emph{simplicial} if all of its facets contain exactly $d$ vertices. 
    It is \emph{simple} if each of its vertices is contained in exactly $d$ facets.
    A polytope is \emph{$k$-neighborly} if any set of $k$ vertices is a face, $f_{k-1}=\binom{f_0}{k}$. 
    A $d$-polytope is \emph{neighborly} if it is $k$-neighborly for all $k\leq \lfloor\frac{d}{2}\rfloor$.
    Two polytopes are called \emph{combinatorially equivalent} if they have isomorphic face lattices.
    A $d$-polytope is \emph{inscribed} if all its vertices lie on the unit $(d-1)$-sphere, i.e.~if $\sum_{i=1}^{d}v_i^2 = 1$ for each vertex $v=(v_1,\dots,v_d)$. If $P$ is combinatorially equivalent to an inscribed polytope it is \emph{inscribable}.
 \end{definition}
 
   \begin{definition}[simplicial complex, triangulation, simplicial sphere]
    A homeomorphism from the geometric realization of a simplicial complex $\mathcal{C}$ to a topological space $X$ is a \emph{triangulation of $X$}. 
    If $X$ is a sphere, we call the triangulation \emph{simplicial sphere}.
    \end{definition}
    The boundary of a simplicial polytope gives rise to a simplicial sphere. 
   \begin{definition}[polytopal]
     A simplicial sphere $|\mathcal{C}|\to S^{d-1}$ is  \emph{polytopal} if it arises from the boundary of a simplicial $d$-polytope $P$, i.e.~if $\mathcal{C}$ is isomorphic to the set of faces in the boundary of $P$. 
   \end{definition}
   \begin{definition}[prerequisites for covectors]
    Let $E$ be a finite set. A \emph{sign vector} is an element $C\in \{-,0,+\}^E$. Given two sign vectors $C$ and $D$, we define their composition as
    \[(C\circ D)_e\colonequals\begin{cases}C_e &\text{ if }C_e\neq0\\
                                            D_e &\text{ otherwise }
                          \end{cases}\]
    for $e\in E$.
    An element $e\in E$ \emph{separates} $C$ and $D$, if $0\neq C_e=-D_e$. The set of all elements which separate $C$ and $D$ is denoted by $S(C,D)$.
   \end{definition}

   \begin{definition}[oriented matroid given by covectors]
     An \emph{oriented} matroid is given by a finite set $E$ together with a set $\mathcal{L}\subset\{-,0,+\}^E$ of \emph{covectors}, which satisfies
     \begin{enumerate}[\rm i)]
      \item $\mathbf{0}\in\mathcal{L}$,
      \item if $C\in \mathcal{L}$ then $-C\in\mathcal{L}$,
      \item if $C$ and $D \in \mathcal{L}$, then $C\circ D\in\mathcal{L}$,
      \item if $C,D \in \mathcal{L}$ and $e\in S(C,D)$,\newline then there is $Z\in\mathcal{L}$ such that $Z_e=0$ and $Z_f=(C\circ D)_f$ for all $f\in E\setminus S(C,D)$. 
     \end{enumerate}
   \end{definition}
   
   \begin{definition}[rank, uniform]
     The \emph{rank} of the oriented matroid is defined as the rank of the underlying matroid. 
     It is called \emph{uniform} if the underlying matroid is uniform
   \end{definition}
   \begin{definition}[acyclic, face lattice, matroid polytope]
    An oriented matroid $\mathcal{M}=(E,\mathcal{L})$ is called \emph{acyclic} if \[(+,\dots,+)\in\mathcal{L}\]
    The set of \emph{faces} of $\mathcal{M}$ is defined as
     \[FL(\mathcal{M})\colonequals\left\{{C^0 \, |\, C\in \mathcal{L}\cap \{0,+}\}^E\right\}\]
     This set is partially ordered by inclusion and is called the \emph{face lattice} of $\mathcal{M}$. 
    An acyclic oriented matroid is \emph{$k$-neighborly} if any set of $k$ elements of $E$ is a face. 
    An acyclic oriented matroid of rank $r$ is \emph{neighborly} if it is $k$-neighborly for all $k\leq \lfloor\frac{r-1}{2}\rfloor$.
     The acyclic oriented matroid $\mathcal{M}$ is called a \emph{matroid polytope} if for every $e\in E$, we have $\{e\}\in FL(\mathcal{M}),$ that is if it is $0$-neighborly.
   \end{definition}
   The face lattice of a uniform matroid polytope induces a simplicial sphere.

   \begin{definition}[chirotope]\label{Chi}
    Let $E$ be a finite set and $r$ an integer. A \emph{chirotope} of rank $r$ is a map 
     \[\chi\,: E^{r}\to \{-1,0,1\} \]
     such that
     \begin{enumerate}[\rm i)]
      \item $\chi$ is alternating, i.e.~$\chi\circ\sigma=\sign(\sigma)\chi$ for all permutations $\sigma\in \Sigma_r$. 
      \item\label{GPeqs} For all $\lambda\in E^{r-2}$ and $a,b,d,e\in E\setminus \lambda$ the set 
      \[\left\{\chi(\lambda,a,b)\chi(\lambda,c,d),-\chi(\lambda,a,c)\chi(\lambda,b,d),\chi(\lambda,a,d)\chi(\lambda,b,c)\right\}\]
      is either equal to $\{0\}$ or contains $\{-1,1\}$
      \item The set of elements of $E^r$ that are not mapped to zero by $\chi$ constitutes the basis elements of a matroid, and is in particular non-empty.
     \end{enumerate}
   A chirotope gives rise to an oriented matroid and vice versa. If $0$ is not in the image of $\chi$, we obtain a \emph{uniform} oriented matroid.
   \end{definition}

   \begin{definition}[oriented matroid and chirotope of a configuration of vectors] 
   Given a configuration $X$ of $n$ vectors $p_1,\dots p_n\in  \mathbb{R}^d$ that span $\mathbb{R}^d$, the \emph{oriented matroid} $\mathcal{M}_X\colonequals([n],\mathcal{L}_X)$ of rank $r\colonequals d+1$ 
   is given by 
   \[\mathcal{L}_X\colonequals \left\{\left(\sign(q\cdot \overbar{p}_1),\dots,\sign (q\cdot \overbar{p}_n)\right)\,\middle |\, q \in \mathbb{R}^r\right\},\]
   where $\overbar{p}_m\colonequals\begin{pmatrix}
                                     p_m\\1
                                    \end{pmatrix}$
                                    
   The associated \emph{chirotope} $\chi_X$ of $\mathcal{M}_X$ is the map:
   \begin{align*}\chi_X\,: \{1,\dots,n\}^{r}&\to \{-1,0,1\}\\
	    (m_1,\dots m_{r})&\mapsto \sign\det\left(\overbar{p}_{m_1},\dots,\overbar{p}_{m_r}\right)
   \end{align*} If the points are in \emph{general position}, i.e.~if $0$ is not in the image of $\chi_X$, we obtain an \emph{uniform} oriented matroid.
   \end{definition}
   The oriented matroid and chirotope of a configuration of vectors is indeed an oriented matroid and chirotope; in fact the property of the former inspire the definition of the latter.
   If the point configuration is the set of vertices of a simplicial polytope $P$, we will obtain a matroid polytope $\mathcal{M}$ and the face lattice of $P$ will be isomorphic to the face lattice of $\mathcal{M}$.
   \begin{definition}
     A chirotope (resp. oriented matroid) is \emph{realizable} if it can be obtained as the chirotope (resp. oriented matroid) of a configuration of vectors.  
   \end{definition}

 \subsection{Finding realizations and inscriptions as an optimization problem}\label{opti}
  
 Let $\chi_\mathcal{M}$ be a chirotope of a uniform matroid polytope $\mathcal{M}$ of rank $r$ with $n$ elements. To the chirotope we associate the following system of polynomial inequalities: 
 \begin{align}\label{eins}
  \chi_\mathcal{M}(m_1,\dots,m_r)\det\left(\overbar{p}_{m_1},\dots,\overbar{p}_{m_r}\right)> 0&\text{ for all }m_1,\dots,m_r\in\binom{[n]}{r}
 \end{align}
Here $p_m\colonequals(p_{m,1},\dots,p_{m,r-1})$ for $1\leq m\leq n$ are vectors of real variables, so there are $nm$ many variables.
The system is defined over $\mathbb{R}[p_{m,i} \text{ for }1\leq m\leq n \text{ and }1\leq i \leq r-1]$ with $\binom{n}{r}$ homogeneous inequalities of degree $r$.
\begin{proposition}
 A uniform matroid polytope $\mathcal{M}$ is realizable if and only if the system $(\ref{eins})$  has a solution.
\end{proposition}
Inequality (\ref{eins}) implies 
\[\chi(m_1,\dots m_{r})=\sign\det\left(\overbar{p}_{m_1},\dots,\overbar{p}_{m_r}\right),\]
which is just what is needed in the definition of chirotope of a configurations of vectors. 

In addition, we could ask for all vertices to lie on the unit sphere:
\begin{align}\label{drei}
 \sum_{i=1}^{r-1}p_{m,i}^2=1 &\text{ for }1\leq m \leq n.
\end{align}
\begin{proposition}
 The uniform matroid polytope $\mathcal{M}$ is realizable as an inscribed polytope if and only if the system $(\ref{eins})$ and $(\ref{drei})$ has a solution.
\end{proposition}
We can weaken $(\ref{eins})$ and only consider inequalities that concern faces of $\mathcal{M}$:
 \begin{align*}
  \chi_\mathcal{M}(m_1,\dots,m_r)\det\left(\overbar{p}_{m_1},\dots,\overbar{p}_{m_r}\right)>0&\text{ for all }m_1,\dots,m_r\in\binom{[n]}{r}\\
  \text{ if for some }j, (m_1,\dots,\widehat{m_j},\dots,m_r)&\text{ is a face of }M\numberthis \label{vier}
 \end{align*}
\begin{proposition}
 The face lattice of $\mathcal{M}$ is polytopal if and only if the system $(\ref{vier})$ and has a solution.
\end{proposition}
\begin{proposition}
 The face lattice of $\mathcal{M}$ is the face lattice of an inscribable polytope if and only if the system $(\ref{vier})$ and $(\ref{drei})$ has a solution.
\end{proposition}
In the last two propositions only the partial information of the chirotope is required, namely the orientation of the simplices that contain a face. 
It is straightforward generate this partial information when given a simplicial complex.

 In order to solve such systems of inequalities and equations,  we need a solver for nonlinear programs. 
 For our computations we have used SCIP, which 
uses branch and bound techniques and linear underestimation in order to find a feasible solution within a certain precision; see \cite{A09} and \cite{ABKW08} for details.

For numerical reasons, the solver cannot handle strict inequalities, which is why we adapt inequalities (\ref{eins}) by replacing ``$>0$'' by ``$\geq \varepsilon$'' for some small positive $\varepsilon$. 
If we simply replace the strict inequalities by weak inequalities or if epsilon is very close to machine precision, we will obtain trivial solutions. 
It is important to choose $\varepsilon$ adequately.
We can be certain that for our results no feasible solutions have been discarded, since we prove the infeasibility of the system using different techniques, see Section \ref{nonreal}.

In all cases discussed, a numerical solution with a certain precision can be turned into a rational solution of the system in question, see Section \ref{rational}. 
Once we have a solution with rational coordinates we can prove the validity of the systems of (in)equalities (\ref{eins}), (\ref{drei}) and (\ref{vier}) by calculation in exact arithmetic.

That procedure works reasonably well in practice for \emph{finding} realizations if they exist. 
If there is no realization and the system of inequalities and equations is therefore infeasible, the optimizer does not terminate in a reasonable amount of time or runs out of memory. 
See Section \ref{nonreal} on how to handle potentially non-realizable cases.

 \subsection{Finding rational points on the sphere}\label{rational}
 From the calculations described in the previous section, we obtain \emph{numerical} solutions of the system of inequalities and equations, which are not guaranteed to be correct.
 The goal is to derive \emph{rational} points from these solutions that satisfy the system of inequalities and equations in exact arithmetic.
 This is in particularly easy if we investigate the realization of \emph{uniform} matroids that are realized as \emph{simplicial} polytopes, whose combinatorial type is unchanged by a small distortion.
 When we look at inscribed realizations, we start with a point $x$ given numerically that is very close to the unit sphere and we are looking for a point $x'$ with rational coordinates \emph{on} the sphere very close to $x$. 
 Since the rational points are dense in the unit sphere, the existence of a good rational approximation is guaranteed. 
 Consider a rational line through the sphere that intersects the unit sphere in two points.
 If one of the points is rational, let's say it is the north pole, then the other intersection point will also be rational. 
 We notice that stereographic projections and its inverse send rational points to rational points and is continuous away from the projection point. 
 
 This enables us to find a suitable rational point constructively as follows:
 \begin{construction}\nopagebreak~\nopagebreak
 \begin{description}
      \item[Step 1] Use stereographic projection to map $x\in S^d$ to a point $\widetilde{x}\in\mathbb{R}^{d-1}$.
      \item[Step 2] Find a suitable rational approximation $\widetilde{x}'$ for $\widetilde{x}$.
      \item[Step 3] Use the inverse stereographic projection to map $\widetilde{x}'$ to a rational point $x'$ on the sphere.
     \end{description}
  \end{construction}

\subsection{Certificates for non-realizability}\label{nonreal}
\subsubsection{Biquadratic final polynomials}\label{oms}
An oriented matroid is non-realizable if and only if it has a final polynomial, see \cite{BS89}.
Given a (uniform) oriented matroid, there is a good algorithm for \emph{showing non-realizability}, it finds \emph{biquadratic final polynomials} that prove non-realizability.
This is described by Bokowski, Richter and Sturmfels \cite{BR90}, \cite[Sect. 7.3, p. 121]{BS89}.
There are cases of non-realizable oriented matroids that do not possess a \emph{biquadratic} final polynomial,
but do possess a final polynomial (and hence are non-realizable) the first one is given by Richter-Gebert \cite{RG96}. 
We might have found another such instance, see Theorem \ref{4vert} \ref{556061}).

\subsubsection{From simplicial spheres to uniform oriented matroids}\label{sstoom}
Showing that a simplicial sphere is not realizable can be done in two steps:
\begin{enumerate}[\rm i)]
\item\label{seins}generate all compatible uniform matroid polytopes (possibly there aren't any!)
\item\label{szei}find final polynomials for all of them.
\end{enumerate}
Given a simplicial sphere $S$, the values of a compatible chirotope $\chi$ on tuples that contain a face of $S$ are already determined, if we fix the sign of one of those tuples. 
(We can always flip all the signs of a chirotope and obtain a valid chirotope again.) 
All compatible chirotopes are precisely the ones that satisfy the conditions on the signs derived from the Graßmann-Plücker identities (condition \ref{GPeqs}) in Definition \ref{Chi}).
These can be formulated as a boolean satisfiability problem (SAT), compare the work by Schewe \cite{S10}, and has been implemented by David Bremner, see \cite[Sect. 3]{BBG09}.
It can also be formulated as an integer program, which has been done by the author. 
Then a solver for integer programs can be used to generate all compatible matroids.
It might of course be the case, that the system has no solution, which means that the simplicial sphere has no compatible uniform matroid polytopes. 
In the case of odd-dimensional neighborly simplicial spheres, there is at most one solution. 
This property is called rigidity and has been established for polytopes of even dimension in \cite{S82} and, more generally, for neighborly oriented matroids of odd rank in \cite{S88}.
\subsubsection{Using partial chirotopes}\label{partchiro}
In some cases there will be many compatible chirotopes and there might be too many to find (biquadtratic) final polynomials for all of them. 
Sometimes, however, we are still able to prove non-realizability by using only partial information of the chirotopes. 
As mentioned above, all compatible chirotopes for a given simplicial sphere have their values on tuples that contain a face of $S$ in common, if we fix the sign of one of those tuples. 
In general those chirotopes will have values on even more tuples in common, and these values can be determined by examining the conditions on the signs derived from the Graßmann-Plücker identities.
We call this the \emph{partial chirotope} compatible with $S$. 
The question if a partial chirotope can be completed is NP-complete, see \cite{T01} and \cite{B05}.
In many small cases the following approach works reasonably well.
The method of finding biquadratic final polynomials by Bokowski and Richter \cite{BR90},
consists of setting up a linear program that encodes the 3-term Graßmann-Plücker relations using the signs of the chirotopes. If the program is infeasible, then a biquadratic final polynomial exists. 
If the complete chirotope is not known, but only the partial chirotope, we can still set up the linear program, only with less constraints. 
The infeasibility of this program will still prove the existence of a biquadtratic final polynomial.

To summarize, another method for proving that a simplicial sphere is not realizable, without generating all compatible matroid polytopes, is the following:
\begin{enumerate}[\rm i)]
\item\label{ssseins}find a partial chirotope (if there are any compatible uniform matroid polytopes)
\item\label{sszei}find biquadratic final polynomials for it.
\end{enumerate}
\subsection{Computations and hardware}
For the calculations of the results presented in Section \ref{results} the systems of (in)equalities from the simplicial complex and face lattices that are passed to SCIP were set up with the computer algebra system Sage \cite{sage}.
The smaller cases were done on a desktop PC, with 8GB of RAM, the larger cases ran on a cluster on about 300 Xeon CPUs with about 3GB RAM each. 
The time needed for an individual realization varied depending on the dimension and number of points between less than a second and several minutes. 
In the largest cases the size of the system of (in)equalities passed to SCIP were several hundred megabytes. 
Sage was also used to verify the solution in exact arithmetic and to prove non-realizability by an implementation of the biquadratic final polynomial method which runs on the partial chirotope.
 
\section{Results}\label{results}
 Data for all our results are available at the author's web page:
 \begin{center}
\href{http://page.mi.fu-berlin.de/moritz/}{http://page.mi.fu-berlin.de/moritz/}
 \end{center}
  \subsection{Realizations and inscriptions of neighborly polytopes}\label{neighbors}
  Miyata and Padrol \cite{MP15} enumerate simplicial neighborly uniform oriented matroids of various ranks and number of elements.
  This allows us to apply the methods from Section \ref{polys} in order to find realizations of those neighborly uniform oriented matroids. 
  A sewing method of Padrol \cite{P13} provides many neighborly polytopes, which are also inscribable, see \cite{GP15}, and include all simplicial neighborly $d$-polytopes with up to $d+3$ vertices.
  We present results on realizability and inscribabiliy for neighborly simplicial $d$-polytopes with $n$ vertices for  $d=4,5,6$ and $7$  and $n>d+3$.
  \subsubsection{Neighborly 4-polytopes}
  The neighborly $4$-polytopes given by Padrol's sewing construction include the $3$ combinatorial types of neighborly $4$-polytopes with $8$ vertices, described in \cite{GS67}.
  The number combinatorial types of neighborly $4$-polytopes with $n$ vertices was previously only known for  $n\leq 10$, see \href{https://oeis.org/A133338}{A133338}.
  The number of combinatorial types of neighborly $4$-polytopes with $9$ vertices was determined by Altshuler and Steinberg \cite{AS73}, for $10$ vertices it was determined by Altshuler \cite{A77}. 
\begin{theorem}\label{4vert}~
\begin{enumerate}[\rm i)]
 \item All $23$ distinct combinatorial types of neighborly $4$-polytopes with $9$ vertices are inscribable.
 \item\label{n431} All $431$ distinct combinatorial types of neighborly $4$-polytopes with $10$ vertices are inscribable.
 \item There are precisely $13\,935$ distinct combinatorial types of neighborly $4$-polytopes with $11$ vertices. All of these are inscribable.
 \item\label{556061}The number of distinct combinatorial types of neighborly $4$-polytopes with $12$ vertices is $556\,061$ or $556\,062$ and at least $556\,061$ of those are inscribable.
\end{enumerate}
\end{theorem}
\begin{proof}~
 \begin{enumerate}[\rm i)]
  \item[\rm i)-ii)] We provide rational inscribed realizations for all known combinatorial types.
  \item[\rm iii)] Out of the $13\,937$ combinatorial types of neighborly oriented matroids, $2$ admit a biquadratic final polynomial, see \cite[Sect. 4.1.1]{MP15}. 
  For the remaining combinatorial types we provide rational inscribed realizations
  \item[\rm iv)] We analyzed the $556\,144$ combinatorial types of neighborly oriented matroids given by Miyata and Padrol \cite[p. 3]{MP15}: Using methods from Section \ref{opti} we realized all but $83$ cases. 
On those we ran the biquadratic final polynomial method and obtained certificates for non-realizability in $82$ cases. 
The only case left is $\#374225$, which has the following facet list:

$F374225= \{[0\, 1\, 4\, 8]$\, $[0\, 1\, 4\, 10]$\, $[0\, 1\, 8\, 9]$\, $[0\, 1\, 9\, 10]$\, $[0\, 2\, 3\, 
6]$\, $[0\, 2\, 3\, 11]$\, $[0\, 2\, 4\, 9]$\, $[0\, 2\, 4\, 11]$\, $[0\, 2\, 6\, 7]$\, $[0\, 2\, 7\,
9]$\, $[0\, 3\, 6\, 7]$\, $[0\, 3\, 7\, 9]$\, $[0\, 3\, 9\, 11]$\, $[0\, 4\, 5\, 8]$\, $[0\, 4\, 5\,
9]$\, $[0\, 4\, 10\, 11]$\, $[0\, 5\, 8\, 9]$\, $[0\, 9\, 10\, 11]$\, $[1\, 2\, 3\, 4]$\, $[1\, 2\,
3\, 8]$\, $[1\, 2\, 4\, 5]$\, $[1\, 2\, 5\, 8]$\, $[1\, 3\, 4\, 11]$\, $[1\, 3\, 8\, 11]$\, $[1\, 4\,
5\, 8]$\, $[1\, 4\, 10\, 11]$\, $[1\, 6\, 7\, 9]$\, $[1\, 6\, 7\, 10]$\, $[1\, 6\, 8\, 9]$\, $[1\, 6\,
8\, 10]$\, $[1\, 7\, 9\, 10]$\, $[1\, 8\, 10\, 11]$\, $[2\, 3\, 4\, 11]$\, $[2\, 3\, 5\, 6]$\, $[2\,
3\, 5\, 10]$\, $[2\, 3\, 8\, 10]$\, $[2\, 4\, 5\, 6]$\, $[2\, 4\, 6\, 7]$\, $[2\, 4\, 7\, 9]$\, $[2\,
5\, 8\, 10]$\, $[3\, 5\, 6\, 11]$\, $[3\, 5\, 10\, 11]$\, $[3\, 6\, 7\, 11]$\, $[3\, 7\, 9\, 11]$\,
$[3\, 8\, 10\, 11]$\, $[4\, 5\, 6\, 7]$\, $[4\, 5\, 7\, 9]$\, $[5\, 6\, 7\, 8]$\, $[5\, 6\, 8\, 10]$\,
$[5\, 6\, 10\, 11]$\, $[5\, 7\, 8\, 9]$\, $[6\, 7\, 8\, 9]$\, $[6\, 7\, 10\, 11]$\, $[7\, 9\, 10\, 11]\}$

It remains to decide realizability in this case. If one can find a final polynomial, then this would be an example without \emph{biquadratic} final polynomial with a smaller number of vertices, 
than the example provided by Richter-Gebert \cite{RG96}, which has $14$ elements and is of rank~$3$. 
 \end{enumerate}

\end{proof}
\subsubsection{Simplicial neighborly 5-polytopes}
The number of combinatorial types of simplicial neighborly $5$-polytopes with $9$ vertices was determined by Finbow \cite{Fin14} and also in \cite{FMM13}. 
\begin{theorem}~
\begin{enumerate}[\rm i)]
 \item  All $126$ distinct combinatorial types of simplicial neighborly $5$-polytopes with $9$ vertices are inscribable.
 \item  There are precisely $159\,375$  distinct combinatorial types of simplicial neighborly $5$-polytopes  with $10$ vertices. All of these are inscribable.
\end{enumerate}
\end{theorem}
\begin{proof}~
 \begin{enumerate}[\rm i)]
  \item We provide rational inscribed realizations for all known combinatorial types.
  \item Miyata and Padrol give $159\,750$ neighborly uniform oriented matroids of rank $6$ on $10$ elements, one for each combinatorial type of face lattice. 
We realize $159\,375$ of these face lattices, while not paying attention to realizing the specific matroid, and show that they are all inscribable. 
We use partial information of the chirotope coming from the faces, together with the biquadratic final polynomial method to find certificates for non-realizability for additional $189$ face lattices.
For the remaining $186$ cases in addition to the partial information coming from the faces we use the information coming from Graßmann--Plücker relations. 
This allows us to determine sufficiently many signs of the chirotope in order to obtain biquadratic final polynomials, see Section \ref{partchiro}.
 \end{enumerate}
\end{proof}

\subsubsection{Neighborly 6-polytopes}
The number of combinatorial types of simplicial neighborly $6$-polytopes with $10$ vertices was determined by Bokowski and Shemer \cite{BS87}.
\begin{theorem}~
\begin{enumerate}[\rm i)]
\item\label{zehn} All $37$ distinct combinatorial types of neighborly $6$-polytopes with $10$ vertices are inscribable.
\item There are precisely $42\,099$ distinct combinatorial types of neighborly $6$-polytopes with $11$ vertices. All of these are inscribable.
\item\label{2nzehn} There are precisely $4\,523$ simplicial $2$-neighborly $6$-polytopes with $10$ vertices.  All of these are inscribable.
\end{enumerate}
\end{theorem}
\begin{proof}Notice that \ref{zehn}) is included in \ref{2nzehn}). We provide rational inscribed realizations for all combinatorial types.
\end{proof}

\begin{table}[h!]
\begin{center}
\begin{tabular}{c|r|r|r|r|r|r|r|r|}
           &$n=5$&$n=6$&$n=7$&$n=8$&$n=9$&$n=10$&$n=11$		&$n=12$\\\hline
$d=4$ & $1$ &$1$  &$1$  &$3$  &$\underline{23}$ &$\underline{431}$ &$\underline{\mathbf{13\,935}}$&$\underline{\geq{\mathbf{556\,061}}}$\\
	   &	 &	&    &     &	 &	&		     &$\leq\mathbf{556\,062}$\\\hline
$d=5$&  &$1$  &$1$  &$2$      &$\underline{126}$&$\underline{\mathbf{159\,375}}$ &?&?\\\hline
$d=6$&  &     &$1$  &$1$      &$1$     &$\underline{37}$&$\underline{\mathbf{42\,099}}$ &?\\\hline
$d=7$&  &     &     &$1$      &$1$     &$4$             &$\underline{\mathbf{35\,993}} $     &?\\\hline
\end{tabular}
\end{center}\caption{The numbers of combinatorial types of neighborly simplicial $d$-polytopes with $n$ vertices. \textbf{Boldface} results are new. 
All polytopes in the classes enumerated in this table can be inscribed. 
\underline{Underlined} results about inscribability are new.}\label{tab2}
\end{table}

\subsubsection{Simplicial neighborly 7-polytopes}
\begin{theorem}
There are precisely $35\,993$ distinct combinatorial types of simplicial neighborly $7$-polytopes with $11$ vertices.  All of these are inscribable.
 \end{theorem}
 \begin{proof} We provide rational inscribed realizations for all known combinatorial types.
 \end{proof}
\subsubsection{Summary}
We summarize the results in Table \ref{tab2}, compare \cite[Table 3]{MP15}.

\subsection{Simplicial 4-polytopes with 10 vertices.}\label{simplicial}
    The number of simplicial $3$-polytopes with $n$ vertices is known for $n\leq23$, see \href{https://oeis.org/A000109}{A000109}. 
    Because of the connection with planar graphs it is easier to classify simplicial $3$-polytopes than simplicial $4$-polytopes.
    The number of simplicial $4$-polytopes with $n$ vertices was previously known only for $n\leq9$, see \href{https://oeis.org/A222318}{A222318}, see \cite{ABS80} \cite{GS67} \cite{FMM13}.
    The number of triangulations of $S^3$ is known for $n\leq 10$.
    
 \begin{table}[h]
\centering
    \begin{tabular}{r|r|r|r}
     vertices & triangulations of $S^3$&non-polytopal&polytopal\\\hline
     $5$&$1$&$1$&$0$\\
     $6$&$2$&$2$&$0$\\
     $7$&$5$&$5$&$0$\\
     $8$&$39$&$2$&$37$\\
     $9$&$1\,296$&$154$&$1\,142$\\
     $10$&$247\,882$&$\mathbf{85\,878}$&$\mathbf{162\,004}$
    \end{tabular}
    \caption{polytopal and non-polytopal simplicial $3$-spheres. \textbf{Boldface} results are new.}
  \end{table}

    Frank Lutz gives a complete enumeration of all combinatorial $3$-manifolds with $10$ vertices in \cite{L08}. 
    He finds precisely $247\,882$ triangulations of $S^3$ and asked for the number of simplicial polytopes with $10$ vertices \cite[Prob. 4]{L08}:    
    ``Classify all simplicial $3$-spheres with $10$ vertices into polytopal and non-polytopal spheres.''

      \begin{table}[h]
\centering
\begin{tabular}{c|r|r|r|r|r}
     $f$-vector& $S^3$&non-polytopal&polytopal&inscribable&p., non-i.\footnote{polytopal, but non-inscribable}\\\hline
$(10, 30, 40, 20)$ & $30$ & $ 	0$ &	$30 $&	$27$ &	$3$ \\
$(10, 31, 42, 21)$ & $124$ & $ 	0$ &	$124$&$\leq 119$\hfill and \hfill$\geq\underline{118}$&$\geq 5$	and $\leq \underline{6}$\\
$(10, 32, 44, 22)$ & $385$ & $ 	0$ &	$385$&$\leq 381$ \hfill and \hfill$\geq\underline{379}$ &$\geq 4$	and $\leq \underline{6}$\\
$(10, 33, 46, 23)$ & $952$ & $ 	0$ &	$952$ &	$\leq951$\hfill  and \hfill $\geq\underline{948}$&	$\geq 1$ and $\leq\underline{4} $\\
$(10, 34, 48, 24)$ & $2\,142$ & $ 0 $&	$2\,142$ &$\leq 2\,142$ \hfill and\hfill 	$\geq\underline{2\,139} $&$\geq0$ and $\leq\underline{3} $\\
$(10, 35, 50, 25)$ & $4\,340$ & $ 28 $&	$4\,312$ &$\leq4\,312$ and $\geq\underline{4\,309} $&$\geq 0$ and $\leq\underline{3} $\\
$(10, 36, 52, 26)$ & $8\,106$ & $ 151 $&	$7\,955 $&$\leq 7\,955$ and $\geq\underline{7\,954} $&$\geq 0$ and $\leq\underline{1}$\\
$(10, 37, 54, 27)$ & $13\,853$ & $583 $&	$13\,270 $&	$13\,270 $&	$0 $\\
$(10, 38, 56, 28)$ & $21\,702$ & $1\,862$ &	$19\,840 $&	$19\,840 $&	$0 $\\
$(10, 39, 58, 29)$ & $30\,526$ & $4\,547 $&	$25\,979 $&	$25\,979 $&	$0 $\\
$(10, 40, 60, 30)$ & $38\,553$ & $9\,267 $&	$29\,286 $&	$29\,286 $&	$0 $\\
$(10, 41, 62, 31)$ & $42\,498$ & $15\,680$ &$26\,818$ &	$26\,818$ &$	0 $\\
$(10, 42, 64, 32)$ & $39\,299$ & $20\,645$ &$18\,654 $&	$18\,654$ &$	0 $\\
$(10, 43, 66, 33)$ & $28\,087$ & $19\,027$ &$9\,060 $&	$9\,060$ &$	0 $\\
$(10, 44, 68, 34)$ & $13\,745$ & $ 	10\,979$ &$2\,766 $&	$2\,766$ &$	0 $\\
$(10, 45, 70, 35)$ & $3\,540$ & $ 	3\,109 $&	$431$ &	$431$ &$	0 $\\\hline
$(10,\;\,*,\;\,*,\;\,*\,)$   	& $247\,882$ & $ 	85\,878 $&$162\,004$ &	$\geq\underline{161\,978}$ &$\leq\underline{26}$\\
&&&&$\leq161\,991$&$\geq 13$
\end{tabular}
\caption{Simplicial $3$-spheres with $10$ vertices. Conjectured tight bounds are underlined.}
 \label{3spheres}
\end{table}
 \begin{theorem}\label{th23}~
 \begin{enumerate}[\rm i)]
  \item There are precisely $162\,004$ distinct combinatorial types of simplicial $4$-polytopes with $10$ vertices.
  \item\label{zwei} There are precisely $161\,978+D$ distinct combinatorial types of inscribable simplicial $4$-polytopes with $10$ vertices, for some $0\leq D\leq 13$. 
  \item All combinatorial types of simplicial $4$-polytopes with up to $8$ vertices are inscribable.
  \item\label{teilvier} Out of the $1142$ combinatorial types of $4$-polytopes with $9$ vertices, precisely $1140$ are inscribable.
\end{enumerate}
  \end{theorem}\newpage
  \begin{proof}~
 \begin{enumerate}[ i)]\item
    Previously it was known that out of the $247\,882$ triangulations of $S^3$, at least $135\,317$ are polytopal and at least $85\,638$ are non-polytopal. 
    The last number is largely due to David Bremner. 
    He used his program ``matroid polytope completion'' (mpc), see \cite[Sect. 3]{BBG09} to find matroids for these triangulations. 
    If there are no compatible matroids for a given triangulation, this triangulation cannot be polytopal. 
    This way he could sort out $85\,636$ cases. 
    Two additional non-realizable cases are the following:
    there is one non-realizable neighborly triangulation of $S^3$ with $10$ vertices, see Theorem \ref{4vert} \ref{n431}) and
    there is one non-realizable triangulation of $S^3$ with $f$-vector $(1,10,40,60,30)$, which is discussed under the name $T2766$ in Section \ref{smallvalence}.
    We realized $162\,004$ with methods described in Section \ref{opti}.
    For the $240$ remaining cases, we applied the methods from Section \ref{nonreal}. In all but one case we could prove the existence of a biquadtratic final polynomial by using only partial chirotopes. 
    In the remaining case, $12418$ in Lutz's numbering, we generated all $2\,985$ compatible chirotopes and found biquadtratic final polynomials for all of them.

\footnotetext{polytopal, but non-inscribable}

\item We expect $D$ to be zero.
We could inscribe all but $26$ cases of the $162\,004$ realizable simplicial spheres, 
and use the criterion for the inscribability for stacked polytopes given by Gonska and Ziegler \cite[Th. 1]{GZ13} to show $3$ of the combinatorial types of polytopes with $f$-vector
$(10,30,40,20)$ are not inscribable.
Proving that the other cases cannot be inscribed is work in progress; Hao Chen provides a proof of the non-inscribability of $10$ additional cases (private communication).

In Lutz's numbering, the remaining $13$ cases are: $2458$, $7037$, $8059$, $8062$, $116369$, $116370$, $116407$, $116434$, $116437$, $134098$, $136359$, $136366$, $136376$.
\item We found rational coordinates on the sphere with methods described in Section \ref{opti}
\item Here these methods provide rational inscribed realizations for all but $2$ out of the $1\,142$ distinct combinatorial types. 
We could realize the other two combinatorial types, but not with all vertices on the sphere.
One case is the $4$-simplex, stacked on $4$ of its faces; this is non-inscribable because of the criterion given by Gonska and Ziegler.
The other case is constructed as follows: Take the direct sum of two triangles and choose a vertex $v$ in it. Then stack on all three facets that do not contain $v$. 
A proof of the non-inscribability of this polytope with the name $(9,355)$ was provided by Hao Chen (private communication). He uses methods similiar to those used by Gonska and Ziegler \cite{GZ13}.

\footnotetext{polytopal, but non-inscribable}
We observe that these two polytopes are the only $4$-polytopes with $9$ vertices whose edge-graphs have an independent set of size $4$; all the other $1140$ polytopes have maximal independent sets of smaller size.
    \end{enumerate}
 \end{proof}
     The results are summarized in Tables \ref{3spheres} and \ref{3spheres9}.
    \begin{table}[h!]
 \centering
\begin{tabular}{c|r|r|r|r|r}
     $f$-vector& $S^3$&non-polytopal&polytopal&inscribable&\addtocounter{footnote}{-1}\addtocounter{Hfootnote}{-1}p., non-i.\footnote{polytopal, but non-inscribable}\\\hline
     $(5, 10, 10, 5)$ & $1$ & $0$ & $1$ & $1$&$0$\\\hline\hline
     $(6, 14, 16, 8)$ & $1$ & $0$ & $1$ & $1$&$0$\\
$(6, 15, 18, 9)$ & $1$ & $0$ & $1$ & $1$&$0$\\\hline
$(6,\;\,*,\;\,*,\;*)$&$2$&$0$&$2$&$2$&$0$\\\hline\hline
     $(7, 18, 22, 11)$ & $1$ & $0$ & $1$ & $1$&$0$\\
$(7, 19, 24, 12)$ & $2$ & $0$ & $2$ & $2$&$0$\\
$(7, 20, 26, 13)$ & $1$ & $0$ & $1$ & $1$&$0$\\
$(7, 21, 28, 14)$ & $1$ & $0$ & $1$ & $1$&$0$\\\hline
$(7,\;\,*,\;\,*,\;\,*\,)$&$5$&$0$&$5$&$5$&$0$\\\hline\hline
     $(8, 22, 28, 14)$ & $3$ & $0$ & $3$ & $3$&$0$\\
$(8, 23, 30, 15)$ & $5$ & $0$ & $5$ & $5$&$0$\\
$(8, 24, 32, 16)$ & $8$ & $0$ & $8$ & $8$&$0$\\
$(8, 25, 34, 17)$ & $8$ & $0$ & $8$ & $8$&$0$\\
$(8, 26, 36, 18)$ & $6$ & $0$ & $6$ & $6$&$0$\\
$(8, 27, 38, 19)$ & $5$ & $1$ & $4$ & $4$&$0$\\
$(8, 28, 40, 20)$ & $4$ & $1$ & $3$ & $3$&$0$\\\hline
$(8,\;\,*,\;\,*,\;\,*\,)$&$39$&$2$&$37$&$37$&$0$\\\hline\hline
$(9, 26, 34, 17)$ & $7$ & 0 & $7$ & $6$&$1$\\
$(9, 27, 36, 18)$ & $23$ & 0 & $23$ & $22$&$1$\\
$(9, 28, 38, 19)$ & $45$ & 0 & $45$ & $45$&$0$\\
$(9, 29, 40, 20)$ & $84$ & 0 & $84$ & $84$&$0$\\
$(9, 30, 42, 21)$ & $128$ & 0 & $128$ & $128$&$0$\\
$(9, 31, 44, 22)$ & $175$ & 3 & $172$ & $172$&$0$\\
$(9, 32, 46, 23)$ & $223$ & 11 & $212$ & $212$&$0$\\
$(9, 33, 48, 24)$ & $231$ & 22 & $209$ & $209$&$0$\\
$(9, 34, 50, 25)$ & $209$ & 46 & $163$ & $163$&$0$\\
$(9, 35, 52, 26)$ & $121$ & 45 & $76$ & $76$&$0$\\
$(9, 36, 54, 27)$ & $51$ & 28 & $23$ & $23$&$0$\\\hline
$(9,\;\,*,\;\,*,\;\,*\,)$&$1296$&$154$&$1\,142$&$1\,140$&$2$\\
\end{tabular}
\caption{Simplicial $3$-spheres with $\leq 9$ vertices.}
 \label{3spheres9}
\end{table}

   \subsection{Manifolds with small valence}\label{smallvalence}

  We consider  a combinatorial analogue of curvature or angular defect. 
  \begin{definition}[valence] We call a  $(d-2)$-face of a $d$-dimensional simplicial complex a \emph{subridge}. 
  The \emph{valence} of a subrigde is the number facets it is contained in. 
  \end{definition}
  Frick, Lutz and Sullivan consider simplicial manifolds with small valence \cite{FLS15, Fr15}. One case they study in particular are $3$-dimensional manifolds with valence less or equal than $5$. 
  The result of a computer enumeration by Lutz \cite[Th. 3.10]{Fr15}: Out of the $4\,787$ distinct combinatorial types of $3$-dimensional manifolds, there are $4\,761$ triangulations of the $3$-sphere~$S^3$.

  Previously little was known about the polytopality of most of these spheres. In contrast to Section \ref{simplicial} the number of vertices of these triangulation can be larger than $10$, it can be in fact be as large as $120$.
 We were able to realize as polytopes, and even better to find inscriptions on the sphere with rational coordinates, for all but $2$ of those triangulations of $S^3$. 
 \begin{theorem}\label{thsmvalence}
  Out of the $4\,761$ simplicial $3$-spheres with small valence at least $4\,759$ are realizable and inscribable.
 \end{theorem}
 The triangulations, for which we could not find realizations, are $T2766$ and $T2775$ in the numbering used in \cite{Fr15}. The triangulations are given as follows:
 
$T2766=\{[0\, 1\, 2\, 3]$\, $[0\, 1\, 2\, 4]$\, $[0\, 1\, 3\, 5]$\, $[0\, 1\, 4\, 6]$\, $[0\, 1\, 5\, 6]$\, $[0\, 2\, 3\, 4]$\, $[0\, 3\, 4\, 7]$\, 
       $[0\, 3\, 5\, 7]$\, $[0\, 4\, 6\, 7]$\, $[0\, 5\, 6\, 8]$\, $[0\, 5\, 7\, 8]$\, $[0\, 6\, 7\, 8]$\, $[1\, 2\, 3\, 9]$\, $[1\, 2\, 4\, 8]$\, 
       $[1\, 2\, 8\, 9]$\, $[1\, 3\, 5\, 6]$\, $[1\, 3\, 6\, 9]$\, $[1\, 4\, 6\, 9]$\, $[1\, 4\, 8\, 9]$\, $[2\, 3\, 4\, 7]$\, $[2\, 3\, 7\, 9]$\, 
       $[2\, 4\, 7\, 8]$\, $[2\, 5\, 7\, 8]$\, $[2\, 5\, 7\, 9]$\, $[2\, 5\, 8\, 9]$\, $[3\, 5\, 6\, 9]$\, $[3\, 5\, 7\, 9]$\, $[4\, 6\, 7\, 8]$\, 
       $[4\, 6\, 8\, 9]$\, $[5\, 6\, 8\, 9]\}$

 $T2775=\{[0\, 1\, 2\, 3]$\, $[0\, 1\, 2\, 4]$\, $[0\, 1\, 3\, 5]$\, $[0\, 1\, 4\, 6]$\, $[0\, 1\, 5\, 6]$\, $[0\, 2\, 3\, 4]$\, $[0\, 3\, 4\, 7]$\,
	 $[0\, 3\, 5\, 7]$\, $[0\, 4\, 6\, 7]$\, $[0\, 5\, 6\, 8]$\, $[0\, 5\, 7\, 8]$\, $[0\, 6\, 7\, 8]$\, $[1\, 2\, 3\, 9]$\, $[1\, 2\, 4\, 10]$\, 
	  $[1\, 2\, 9\, 10]$\, $[1\, 3\, 5\, 11]$\, $[1\, 3\, 9\, 11]$\, $[1\, 4\, 6\, 12]$\, $[1\, 4\, 10\, 12]$\, $[1\, 5\, 6\, 12]$\,
	  $[1\, 5\, 11\, 12]$\, $[1\, 9\, 10\, 11]$\, $[1\, 10\, 11\, 12]$\, $[2\, 3\, 4\, 12]$\, $[2\, 3\, 9\, 12]$\, $[2\, 4\, 10\, 12]$\, 
	  $[2\, 9\, 10\, 13]$\, $[2\, 9\, 12\, 13]$\, $[2\, 10\, 12\, 13]$\, $[3\, 4\, 6\, 7]$\, $[3\, 4\, 6\, 12]$\, $[3\, 5\, 7\, 11]$\, 
	  $[3\, 6\, 7\, 11]$\, $[3\, 6\, 9\, 11]$\, $[3\, 6\, 9\, 12]$\, $[5\, 6\, 8\, 9]$\, $[5\, 6\, 9\, 12]$\, $[5\, 7\, 8\, 13]$\, 
	  $[5\, 7\, 11\, 13]$\, $[5\, 8\, 9\, 13]$\, $[5\, 9\, 12\, 13]$\, $[5\, 11\, 12\, 13]$\, $[6\, 7\, 8\, 11]$\, $[6\, 8\, 9\, 11]$\, 
	  $[7\, 8\, 11\, 13]$\, $[8\, 9\, 10\, 11]$\, $[8\, 9\, 10\, 13]$\, $[8\, 10\, 11\, 13]$\, $[10\, 11\, 12\, 13]\}$

Triangulation $T2766$ cannot be realized as a polytope, as explained in \cite{BS95} and \cite{Fr13}.

It remains to show that triangulation $T2775$ is not polytopal, which is what we expect to be the case.

   \subsection{A special inscribed realization for the Bokowski--Ewald--Kleinschmidt polytope}\label{BEK}
   
   Bokowski, Ewald and Kleinschmidt provide a $4$-polytope on $10$ vertices with disconnected realization space, see \cite{BEK84} and \cite{BG90}.
   While enumerating all simplicial $4$-polytopes with $10$ vertices, we also realized this one: it has number $6986$ in Lutz's numbering.
   We provide the following rational coordinates on the sphere:
  \[\begin{array}{c|c|c|c|c|c|c|c|c|c}
  0&1&2&3&4&5&6&7&8&9\\\hline  \renewcommand{\arraystretch}{2.0}
 \begin{array}{r} -\tfrac{20}{583}  \\  \tfrac{2}{53}  \\  \tfrac{38}{583}  \\  \tfrac{581}{583} \end{array}&\renewcommand{\arraystretch}{2.0}
 \begin{array}{r} -\tfrac{2}{17}  \\  \tfrac{16}{51}  \\  -\tfrac{10}{51}  \\  \tfrac{47}{51} \end{array}&\renewcommand{\arraystretch}{2.0}
 \begin{array}{r} -\tfrac{6}{61}  \\  \tfrac{20}{61}  \\  \tfrac{6}{61}  \\  \tfrac{57}{61} \end{array}&\renewcommand{\arraystretch}{2.0}
 \begin{array}{r} -\tfrac{5}{18}  \\  -\tfrac{1}{6}  \\  -\tfrac{1}{18}  \\  \tfrac{17}{18} \end{array}&\renewcommand{\arraystretch}{2.0}
 \begin{array}{r} -\tfrac{4}{237}  \\  \tfrac{8}{79}  \\  -\tfrac{56}{237}  \\  \tfrac{229}{237} \end{array}&\renewcommand{\arraystretch}{2.0}
 \begin{array}{r} \tfrac{10}{59}  \\  \tfrac{10}{59}  \\  \tfrac{16}{59}  \\  \tfrac{55}{59} \end{array}&\renewcommand{\arraystretch}{2.0}
 \begin{array}{r} \tfrac{4}{79}  \\  -\tfrac{80}{553}  \\  \tfrac{40}{553}  \\  \tfrac{545}{553} \end{array}&\renewcommand{\arraystretch}{2.0}
 \begin{array}{r} \tfrac{4}{27}  \\  -\tfrac{40}{189}  \\  -\tfrac{8}{63}  \\  \tfrac{181}{189} \end{array}&\renewcommand{\arraystretch}{2.0}
 \begin{array}{r} \tfrac{28}{79}  \\  -\tfrac{4}{79}  \\  -\tfrac{20}{79}  \\  \tfrac{71}{79} \end{array}&\renewcommand{\arraystretch}{2.0}
 \begin{array}{r} \tfrac{48}{221}  \\  \tfrac{32}{221}  \\  -\tfrac{12}{221}  \\  \tfrac{213}{221} \end{array}\renewcommand{\arraystretch}{2.0}
     \end{array}\]

    The polytope has $f$-vector $(1, 10, 38, 56, 28)$ and its facet list is: $\{[0\, 1\, 2\, 3]$\, $[0\, 1\, 2\, 4]$\, $[0\, 1\, 3\, 4]$\, $[0\, 2\,
3\, 5]$\, $[0\, 2\, 4\, 9]$\, $[0\, 2\, 5\, 9]$\, $[0\, 3\, 4\, 6]$\,
$[0\, 3\, 5\, 6]$\, $[0\, 4\, 6\, 7]$\, $[0\, 4\, 7\, 9]$\, $[0\, 5\,
6\, 9]$\, $[0\, 6\, 7\, 9]$\, $[1\, 2\, 3\, 8]$\, $[1\, 2\, 4\, 9]$\,
$[1\, 2\, 5\, 8]$\, $[1\, 2\, 5\, 9]$\, $[1\, 3\, 4\, 8]$\, $[1\, 4\,
8\, 9]$\, $[1\, 5\, 8\, 9]$\, $[2\, 3\, 5\, 8]$\, $[3\, 4\, 6\, 7]$\,
$[3\, 4\, 7\, 8]$\, $[3\, 5\, 6\, 7]$\, $[3\, 5\, 7\, 8]$\, $[4\, 7\,
8\, 9]$\, $[5\, 6\, 7\, 8]$\, $[5\, 6\, 8\, 9]$\, $[6\, 7\, 8\, 9]\}$

    We can stack over the facet $[0\, 2\, 5\, 9]$ with the point  $(\tfrac{2}{23} , \tfrac{5}{23} , \tfrac{4}{23} , \tfrac{22}{23} )$  and over another facet $[0\, 3\, 5\, 6]$ with the point
$( \tfrac{4}{203} , -\tfrac{80}{609} , \tfrac{8}{87} , \tfrac{601}{609} )$.
   These two facets lie in the same orbit of the involution given by the permutation $(1,7)(2,6)(3,9)$, such that the stacking points are also on the sphere. 
   This gives rise to two configurations of $10$ points in $\mathbb{R}^3$ which have combinatorial equivalent Delaunay triangulations, but lie in distinct components of its realization space.

    \begin{theorem}[{\cite[Cor. 4.18]{APT15}}]There is a $3$-dimensional configuration of 10 points whose Delaunay triangulation
has a disconnected realization space.
    \end{theorem}

  \section*{Acknowledgements}
  I am indebted to Arnau Padrol, who started the project by telling me about the problem discussed in Section \ref{BEK} and proposing to use optimization methods in order to solve it and to 
  Hao Chen, who got very interested in proving non-inscribability results.
  I am thankful to Hiroyuki Miyata for the help while working on Section \ref{neighbors}; 
  Hiroyuki Miyata was kind enough to run the biquadratic final polynomial method for us until Arnau Padrol and I had implemented it and to run double checks afterwards.
  For the help in Section \ref{simplicial} I thank Frank Lutz, who provided the list of triangulations and David Bremner, who decided the existence of a compatible matroid polytope for these triangulations.
  I enjoyed discussions on the topic discussed in Section \ref{smallvalence}  with Florian Frick, John Sullivan and Frank Lutz, who again provided a list of triangulations.
  I would like to thank Moritz W. Schmitt, Philip Brinkmann and Louis Theran for valuable discussions and Benjamin Müller as well as Ambros Gleixner for help with my questions regarding SCIP.  
    Last but not least, I am very grateful to Günter M. Ziegler for helping at every step in the preparation of this paper.
  \bibliographystyle{alpha}
\renewcommand\refname{}
\renewcommand\refname{\vskip -1cm}

\section*{References}
  All references of the form A...... are the corresponding sequences in \cite{sloane}.

\bibliography{lit.bib}
    
\end{document}